\documentclass[11pt]{article}

\usepackage{amsfonts,amsmath,latexsym,color,epsfig,hyperref,enumitem, amssymb}
\newtheorem{theorem}{Theorem}

\newtheorem{lemma}{Lemma}

\newenvironment{proof}
      {\medskip\noindent{\bf Proof:}\hspace{1mm}}
      {\hfill$\Box$\medskip}
\def\qed{\ifvmode\mbox{ }\else\unskip\fi\hskip 1em plus 10fill$\Box$}

\input{epsf}
\makeatletter
\def\Ddots{\mathinner{\mkern1mu\raise\p@
\vbox{\kern7\p@\hbox{.}}\mkern2mu
\raise4\p@\hbox{.}\mkern2mu\raise7\p@\hbox{.}\mkern1mu}}
\makeatother

\newcommand{\red}[1]{{\color{red}#1}}

\newcommand{\E}{\mathbb E}
\newcommand{\F}{\mathbb F}

\title{\vspace{-0.7cm}Norm hypergraphs}
\author{Cosmin Pohoata\thanks{Department of Mathematics, Yale University, USA. Email: {\tt andrei.pohoata@yale.edu}.} \and Dmitriy Zakharov\thanks{Laboratory of Combinatorial and Geometric Structures, MIPT, Russia. Email: {\tt  zakharov2k@gmail.com}.}}
\date{}

\begin{document}
\maketitle

\begin{abstract}
We introduce a high uniformity generalization of the so-called (projective) norm graphs of Alon, Koll\'ar, R\'onyai, and Szab\'o, and use it to show that
$$\operatorname{ex}_{d}(n,K_{s_{1},\ldots,s_{d}}^{(d)}) = \Theta\left(n^{d - \frac{1}{s_{1}\ldots s_{d-1}}}\right)$$
holds for all integers $s_{1},\ldots,s_{d} \geq 2$ such that $s_{d} \geq \left((d-1)(s_{1}\ldots s_{d-1}-1)\right)!+1$. This improves upon a recent result of Ma, Yuan and Zhang, and thus settles (many) new cases of a conjecture of Mubayi. %Our construction is inspired in various ways by the celebrated work of Schmidt on norm form equations from 1972.
\end{abstract}

\section{Introduction}

Let $K_{s_{1},\ldots,s_{d}}^{(d)}$ denote the complete $d$-partite $d$-uniform hypergraph with parts of sizes $s_1, \ldots, s_d$ and let $\operatorname{ex}_{d}(n,K_{s_{1},\ldots,s_{d}}^{(d)})$ be the maximum number of hyperedges in a $d$-uniform hypergraph on $n$ vertices which is free of copies of $K_{s_{1},\ldots,s_{d}}^{(d)}$. For $d=2$, the problem of determining $\operatorname{ex}(n,K_{s_{1},s_{2}}):=\operatorname{ex}_{2}(n,K_{s_{1},s_{2}})$ is arguably one of the most famous in combinatorics, the so-called Zarankiewicz problem. Despite considerable interest, the value of this function is known for only a few pairs $(s_{1},s_{2})$. Suppose $G = (V,E)$ is a $K_{s_{1},s_{2}}$-free graph with $s_{1}\leq s_{2}$. The inequality $\sum_{x \in V}{\operatorname{deg}(x) \choose s_{1}} \leq (s_{2}-1){n \choose s_{1}}$ due to K\H{o}vari, S\'os and Tur\'an \cite{KST54} implies the celebrated upper bound 
$$\operatorname{ex}(n,K_{s_{1},s_{2}}) \leq \frac{1}{2}(s_{2}-s_{1}+1)^{1/s_{1}} n^{2-1/s_{1}} + o(n^{2-1/s_{1}}).$$
However, the only cases where the upper bound has been so far matched by a construction with $\Omega(n^{2-1/s_{1}})$ edges are $(s_{1},s_{2}) = (2,2),(2,t),(3,3)$, and, more generally, $(s,t)$ where $t \geq (s-1)!+1$. The latter is a theorem resulting from the combined effort of Alon, Koll\'ar, R\'onyai and Szab\'o~\cite{ARSz99, KRSz96}, which builds upon a long history of earlier work on special cases (see, for example, the comprehensive survey~\cite{FS13}). 

For $d \geq 3$, the situation is even more complicated. Using the double counting idea from \cite{KST54} and induction on $d$, Erd\H{o}s \cite{Erd64} established the following general upper bound:
\begin{equation} \label{erdos}
\operatorname{ex}_{d}(n,K_{s_{1},\ldots,s_{d}}^{(d)}) = O\left(n^{d - \frac{1}{s_{1}\ldots s_{d-1}}}\right)
\end{equation}
for all $s_{1} \leq s_{2} \ldots \leq s_{d}$. Whether or not this result has the chance of always being sharp (up to constants) is a fascinating discussion. In \cite{Mub02}, Mubayi conjectured that this should be indeed the case for all constants $d \geq 2$ and all choices of $s_{1},\ldots,s_{d}$, which is a rather bold claim since not even the case $s_{1}=\ldots=s_{d}=2$ is well-understood. In fact, the problem in this particularly notorious regime is often called \textit{the Erd\H{o}s box problem}, and has generated quite a bit of activity on its own. We refer the reader to \cite{CPZ20} for the (recent) current record and for more background. 

The situation when $s_{d}$ is significantly larger than $s_{1},\ldots,s_{d-1}$ has also proved to be quite tantalizing for many years. Until not too long ago, the only results available were in the rather degenerate case when $s_{1}=\ldots=s_{d-2}=1$, where Mubayi \cite{Mub02} extended the known constructions from the $d=2$ regime. In \cite{MYZ18}, Ma, Yuan and Zhang then provided the first silver lining by proving the following remarkable result.

\begin{theorem} \label{MYZ}
Let $s_{1},\ldots,s_{d-1} \geq 2$ be integers. Then, there exists a positive constant $C$ depending only on $s_{1},\ldots,s_{d-1}$ such that
$$
\operatorname{ex}_{d}(n,K_{s_{1},\ldots,s_{d}}^{(d)}) = \Theta\left(n^{d - \frac{1}{s_{1}\ldots s_{d-1}}}\right)$$
holds for all $s_{d} \geq C$.
\end{theorem}

The proof of Theorem \ref{MYZ} uses the elegant random algebraic method of Bukh from \cite{Buk15}, which was originally developed in \cite{BBK13} (in a more complicated form) in order to provide an alternative argument for the fact that $\operatorname{ex}(n,K_{s_{1},s_{2}}) = \Omega\left(n^{2-\frac{1}{s_{1}}}\right)$ holds when $s_{2}$ is a sufficiently large in terms of $s_{1}$. Nevertheless, just like in the case $d=2$, the bound on $C$ in terms of $s_{1},\ldots,s_{d-1}$ arising from the argument in \cite{MYZ18} is extremely poor (for $d=2$, the bound from \cite{Buk15} is already of tower type where the height of the tower depends linearly on $s_{1}$). It is therefore still quite natural to ask whether for $d \geq 3$ there exist improved constructions that could show that $C$ can be chosen to be a reasonable quantity in terms of $s_{1},\ldots,s_{d-1}$.

In this paper, we address this problem and improve upon Theorem \ref{MYZ} by showing the following result.

\begin{theorem}\label{main}
Let $s_{1},\ldots,s_{d} \geq 2$ be integers such that $s_{d} \geq ((d-1)(s_{1}\ldots s_{d-1}-1))!+1$. Then, there exists a $d$-uniform $K_{s_{1},\ldots,s_{d}}^{(d)}$-free hypergraph $\mathcal{H}(V,E)$ with $|V(\mathcal{H})|=N$ and
$$|E(\mathcal{H})| = \Omega\left(N^{d-\frac{1}{s_{1}\ldots s_{d-1}}}\right).$$
\end{theorem}

This confirms Mubayi's conjecture for all $d \geq 2$ and $s_{d} \geq ((d-1)(s_{1}\ldots s_{d-1}-1))!+1$. Furthermore, it is perhaps worth emphasizing that Theorem \ref{main} also directly recovers the result of Alon, R\'onyai and Szab\'o for the Zarankiewicz problem for graphs \cite{ARSz99}.

\bigskip

\section{Norm hypergraphs}

We first provide a construction which shows that 
\begin{equation}\label{normhyper}
\operatorname{ex}_{d}(n,K_{s_{1},\ldots,s_{d}}^{(d)}) = \Theta\left(n^{d - \frac{1}{s_{1}\ldots s_{d-1}}}\right)
\end{equation}
holds for all integers $s_{1},\ldots,s_{d} \geq 2$ such that $s_{d} \geq ((d-1)s_{1} \ldots s_{d-1})!+1$.
This construction will represent a generalization of the original norm graph introduced by Koll\'ar, R\'onyai, and Szab\'o in \cite{KRSz96}, and is inspired in various ways by the celebrated work of Schmidt on norm form equations from \cite{Sch72}. Our proof will also rely on the following remarkable lemma from \cite{KRSz96}, which we shall now state for the reader's convenience (and ours).

\begin{lemma}
\label{KRS}
Let $\mathbb{F}$ be any field and $a_{ij}, b_i \in \mathbb{F}$ such that $a_{ij} \neq a_{i'j}$ for all $i \neq i'$. Then the system of equations
\begin{align*}
(x_1-a_{11})(x_2-a_{12})\cdots (x_s-a_{1s}) &= b_1 \\
(x_1-a_{21})(x_2-a_{22})\cdots (x_s-a_{2s}) &= b_2 \\
&\vdots \\
(x_1-a_{s1})(x_2-a_{s2})\cdots (x_s-a_{ss}) &= b_s
\end{align*}
has at most $s!$ solutions in $\mathbb{F}^s$.
\end{lemma}

We note that the threshold for $s_{d}$ from \eqref{normhyper} is slightly worse than our Theorem \ref{main}, however it is already a significant improvement of Theorem \ref{MYZ}. We will discuss this first in Section 2.1. In Section 2.2, we will then improve upon this threshold further by appealing to a variant of the ``projectivization" trick due to Alon, R\'onyai and Szab\'o from \cite{ARSz99}.

\subsection{Construction}\label{sec1}

Let $p$ be a prime number, let $m = s_1 \ldots s_{d-1}$, $q = p^m$ and $q' = q^{d-1}$.
Consider the field extensions $\F_p \subset \F_{q} \subset \F_{q'}$. Let $N: \F_{q'} \rightarrow \F_p$ denote the norm map.
Choose elements $\alpha_1, \ldots, \alpha_d \in \F_{q'}$ in such a way that any $d-1$ of them are linearly independent over $\F_q$.

\begin{lemma}\label{cl}
There is a $d$-partite hypergraph $\mathcal F \subset \F_q \times \ldots \times \F_q$ such that $|\mathcal F| \gg q^d$ and such that for any $i = 1, \ldots, d$, any $g \in \operatorname{Gal}(\F_{q'} / \F_q)$, $g \neq 1$, and any edges $(x_1, \ldots, x_d), (y_1, \ldots, y_d) \in \mathcal F$ we have
\begin{equation}\label{eq1}
\sum_{j = 1, j \neq i}^d \frac{\alpha_j}{\alpha_i}  x_j \neq g\left(\sum_{j = 1, j \neq i}^d \frac{\alpha_j}{\alpha_i}  y_j\right).
\end{equation}
\end{lemma}

\begin{proof}
Let $S \subset \F_{q'}$ be the union of all proper subfields of $\F_{q'}$. Note that $|S| = o_{p \to \infty}(q')$. Denote $G = \operatorname{Gal}(\F_{q'} / \F_q)$ and note that any element $x \in \F_{q'} \setminus S$ has a trivial stabilizer under the action of $G$. Let $Y_1, \ldots, Y_d$ be independent uniformly random subsets of $\F_{q'} \setminus S$ such that for any $x \in \F_{q'} \setminus S$ and for every $i = 1, \ldots, d$ we have $|Gx \cap Y_i| = 1$.

For $i = 1, \ldots, d$ and $x_1, \ldots, x_d \in \F_q$ denote $L_i(x_1, \ldots, x_d) = \sum_{j = 1, j \neq i}^d \frac{\alpha_j}{\alpha_i}  x_j$. Let $\mathcal F \subset \F_q^d$ be the random hypergraph consisting of all edges $(x_1, \ldots, x_d)$ such that $L_i(x_1, \ldots, x_d) \in Y_i$ for all $i = 1, \ldots, d$. By design, the family $\mathcal F$ satisfies (\ref{eq1}). So it is left to show that $\E |\mathcal F| \gg q^d$.

Note that the number of tuples $(x_1, \ldots, x_d) \in \F_q^d$ such that for every $i = 1, \ldots, d$ we have $L_i(x_1, \ldots, x_d) \not \in S$ is $(1-o(1))q^d$. On the other hand, for every such $(x_1, \ldots, x_d)$ the probability that $L_i(x_1, \ldots, x_d) \in Y_i$ is equal to $\frac{1}{d-1}$. Therefore, by the independence of $Y_i$, the probability that $(x_1, \ldots, x_d) \in \mathcal F$ is equal to $\frac{1}{(d-1)^d}$. We conclude that 
$$\E|\mathcal F| \ge (1-o(1))\frac{q^d}{(d-1)^d},$$
as claimed. It follows that there exists an $\mathcal{F}$ which satisfies the requirements of Lemma \ref{cl}.
\end{proof}

\bigskip

For $t \in \F_p^*$, let $\mathcal H_t \subset \mathcal F$ be the hypergraph consisting of edges $(x_1, \ldots, x_d)$ such that
\begin{equation} \label{norm2}
N\left(\alpha_{1}x_{1} + \ldots + \alpha_{d} x_{d}\right) = t.
\end{equation}

Observe that the hypergraphs $\mathcal H_t$, $t \in \F_p^*$, cover all the edges of $\mathcal F$. So, by the pigeonhole principle, there exists $t \in \F_p^*$ such that the hypergraph $\mathcal H_t$ has at least $|\mathcal F|p^{-1} \gg q^d p^{-1}$ edges.

We claim that $\mathcal H_t$ does not contain copies of $K^{(d)}_{s_1, \ldots, s_{d-1}, s_d}$ for all $s_d > (m(d-1))!$. Indeed, without loss of generality, for $i = 1, \ldots, d-1$ let $A_i \subset \F_q$ be a subset of size $s_i$. We use Lemma \ref{KRS} to bound the number of elements $x \in \F_q$ such that $(x_1, \ldots, x_{d-1}, x)$ is an edge of $\mathcal H_t$ for all $x_i \in A_i$, $i = 1, \ldots, d-1$. Since in a field with characteristic $p$ we have $(a+b)^p = a^p+b^p$, equation \eqref{norm2} rewrites as
\begin{align}\label{al}
t = N(\alpha_{1}x_{1}+\ldots+\alpha_{d-1}x_{d-1}+\alpha_{d}x) &= \prod_{j=0}^{m(d-1)-1}\left(\alpha_{1}x_{1}+\ldots+\alpha_{d-1}x_{d-1}+\alpha_{d}x\right)^{p^{j}}\\
&= \prod_{j=0}^{m(d-1)-1}\left(\alpha_{1}^{p^{j}}x_{1}^{p^{j}}+\ldots+\alpha_{d-1}^{p^{j}}x_{d-1}^{p^{j}}+\alpha_{d}^{p^{j}}x^{p^{j}}\right).
\end{align}
This holds for all elements $x_{i} \in A_{i}$, as we vary $i \in \left\{1,\ldots,n\right\}$. Note that by (\ref{norm2}) for any $g \in \operatorname{Gal}(\F_{q'}/\F_q)$ we also have
$$
N\left(g(\alpha_{1})x_{1} + \ldots + g(\alpha_{d}) x\right) = t.
$$
Since there are $m$ choices for $x_i \in A_i$ and $d-1$ choices for $g \in \operatorname{Gal}(\F_{q'}/\F_q)$ we obtain a set of $m(d-1)$ equations which $x \in \F_q$ must satisfy. In order to apply Lemma \ref{KRS} we need to check that the coefficients corresponding to the $j$-th bracket in (\ref{al}) are distinct for all choices $x_i \in A_i$ and $g \in \operatorname{Gal}(\F_{q'}/\F_q)$. Since the coefficients in the $j$-th bracket are obtained from the coefficients in the $0$-th bracket by applying an automorphism of $\F_{q'}/\F_p$, it is enough to verify this condition for $j = 0$. 

Suppose that there are $x_i, x'_i \in A_i$, $i = 1, \ldots, d-1$, and $g, g' \in \operatorname{Gal}(\F_{q'}/\F_q)$ such that
\begin{equation}\label{10}
    g\left(\frac{\alpha_1}{\alpha_d}\right) x_1 + \ldots + g\left(\frac{\alpha_{d-1}}{\alpha_d}\right) x_{d-1} = g'\left(\frac{\alpha_1}{\alpha_d}\right) x'_1 + \ldots + g'\left(\frac{\alpha_{d-1}}{\alpha_d}
    \right) x'_{d-1}.
\end{equation}
If $g = g'$ then the fact that the elements $\alpha_i / \alpha_d$ form a basis of $\F_{q'}$ over $\F_q$ implies that $x_i = x'_i$ for all $i = 1, \ldots, d-1$. So suppose that $g \neq g'$. Then (\ref{10}) can be rewritten as follows
\begin{equation}\label{bad}
\frac{\alpha_1}{\alpha_d} x_1 + \ldots + \frac{\alpha_{d-1}}{\alpha_d}x_{d-1} = g^{-1}g' \left(\frac{\alpha_1}{\alpha_d} x'_1 + \ldots + \frac{\alpha_{d-1}}{\alpha_d} x'_{d-1}\right).
\end{equation}
However any solution $x \in \F_q$ to our system of equations in particular satisfies 
$$
(x_1, \ldots, x_{d-1}, x), (x'_1, \ldots, x'_{d-1}, x) \in \mathcal H_t \subset \mathcal F.
$$
Thus, by Lemma \ref{cl}, there is no such $x$ provided that (\ref{bad}) holds.
We conclude that if (\ref{10}) does not hold then the condition of Lemma \ref{KRS} is satisfied and hence there are at most $(m(d-1))!$ solutions to our system of equations. Otherwise, by Lemma \ref{cl}, there are no solutions at all.

\subsection{Improved bound using projectivization}

%\section{New lower bounds for $\operatorname{ex}_{d}(n,K_{s_{1},\ldots,s_{d}}^{(d)})$}

In this section, we prove Theorem \ref{main}. Let $p$ be a prime number, $m = s_1 \ldots s_{d-1}$, $q = p^{m-1}$ and $q' = q^{d-1}$ (note the difference with Section \ref{sec1}). As before, we consider the field extensions $\F_p \subset \F_q \subset \F_{q'}$, define the norm map $N: \F_{q'} \rightarrow \F_p$ and fix elements $\alpha_1, \ldots, \alpha_d \in \F_{q'}$ such that any $d-1$ of them are linearly independent over $\F_q$.

Define a $d$-partite hypergraph $\mathcal H \subset (\F_q \times \F_p^*)^d$ where a sequence of pairs $(x_i, b_i) \in \F_q \times \F_p^*$, $i = 1, \ldots, d$, forms an edge if 
\begin{enumerate}
    \item $(x_1, \ldots, x_d) \in \mathcal F$, where $\mathcal F$ is the family given by Lemma \ref{cl},
    \item $N(\alpha_1 x_1 + \ldots + \alpha_d x_d) = b_1 \ldots b_d$.
\end{enumerate}
Note that $\mathcal H$ has exactly $|\mathcal F| (p-1)^{d-1}$ edges. So to prove Theorem \ref{main} it is enough to show that $\mathcal H$ does not contain copies of $K^{(d)}_{s_1, \ldots, s_{d-1}, s_d}$ for all $s_d > ((m-1)(d-1))!$.
For $i = 1, \ldots, d-1$ fix subsets $A_i = \{ x_{i, 1}, \ldots, x_{i, s_i} \}\subset \F_q$ and $B_i = \{ b_{i, 1}, \ldots, b_{i, s_i} \}\subset \F_p^*$ of size $s_i$. We want to bound the number of pairs $(y, b) \in \F_q \times \F_p^*$ such that 
\begin{equation}\label{inh}
\{(x_{1, j_1}, b_{1, j_1}), \ldots, (x_{d-1, j_{d-1}}, b_{d-1, j_{d-1}}), (y, b) \} \in \mathcal H
\end{equation}
for all $j_1, \ldots, j_{d-1}$. By the definition of $\mathcal H$, (\ref{inh}) implies
\begin{equation}\label{nm}
    N\left(\sum_{i = 1}^{d-1} \frac{\alpha_i}{\alpha_d} x_{i, j_i} + y \right) = N(\alpha_d)^{-1} b_{1, j_1} \ldots b_{d-1, j_{d-1}} b.
\end{equation}

Let $\omega(j_1, \ldots, j_{d-1}) = \sum_{i = 1}^{d-1} \frac{\alpha_i}{\alpha_d} x_{i, j_i}$ and put $z = \frac{1}{y + \omega(1, \ldots, 1)}$. Similarly to (\ref{al}) we can rewrite (\ref{nm}) as follows. Let $g \in \operatorname{Gal}(\F_{q'}/\F_q)$ be an arbitrary element. From (\ref{nm}) we have
\begin{equation}\label{scary}
    \prod_{h \in \operatorname{Gal}(\F_{q'} /\F_p)} (\omega(j_1, \ldots, j_{d-1})^{gh} + y^h) = N(\alpha_d)^{-1} b_{1, j_1} \ldots b_{d-1, j_{d-1}} b.
\end{equation}
Here we use the notation $x^g := g(x)$. Dividing (\ref{scary}) by (\ref{scary}) with $(j_1, \ldots, j_{d-1}) = (1, \ldots, 1)$ and $g= 1 \in \operatorname{Gal}(\F_{q'}/\F_q)$, we obtain
\begin{equation}\label{sad}
    \prod_{h \in \operatorname{Gal}(\F_{q'} /\F_p)} (1 + (\omega(j_1, \ldots, j_{d-1})^{gh} - \omega(1, \ldots, 1)^h)z^h) = \frac{b_{1, j_1} \ldots b_{d-1, j_{d-1}}}{b_{1, 1} \ldots b_{d-1, 1}}.
\end{equation}

So we are in a situation where Lemma \ref{KRS} may be applied. We have $(d-1)(m-1)$ variables $z^h$, $h \in \operatorname{Gal}(\F_{q'}/\F_q)$ and $(d-1)m-1$ equations of the form (\ref{sad}) corresponding to $(j_1, \ldots, j_{d-1}) \neq (1, \ldots, 1)$. If there is at least one pair $(y, b)$ satisfying (\ref{inh}) then, by Lemma \ref{cl}, we have $\omega(j_1, \ldots, j_{d-1})^g \neq \omega(j'_1, \ldots, j'_{d-1})^{g'}$ if $(j_1, \ldots, j_{d-1}, g) \neq (j'_1, \ldots, j'_{d-1}, g')$. Since there are at least as many equations as there are variables, Lemma \ref{KRS} shows that are at most $((d-1)(m-1))!$ elements $y$ satisfying (\ref{sad}). This completes the proof of Theorem \ref{main}.

%\red{Add modified second construction with projectivization trick here.} 

\bigskip

\bigskip

{\bf{Acknowledgements}}. We would like to thank Boris Bukh, David Conlon, and Oliver Janzer for several helpful discussions. The second author would also like to acknowledge the support of the grant of the Russian Government N 075-15-2019-1926.

\end{document}